\documentclass[a4paper]{article}

\usepackage[utf8]{inputenc} 
\usepackage{enumerate}

\usepackage[T1]{fontenc}

\usepackage{hyperref}       
\usepackage{url}            
\usepackage{booktabs}       
\usepackage{amsfonts,amsmath,amssymb,amsthm}       
\usepackage{nicefrac}       
\usepackage{microtype}      
\usepackage{lipsum}
\usepackage{xcolor}
\usepackage{graphicx,float}
\usepackage{tikz}
\usepackage[square,comma,compress,numbers]{natbib}
\usepackage[big]{layaureo}
\usepackage{perpage}
\MakePerPage{footnote}

\newtheorem{theorem}{Theorem}[section]
\newtheorem{lemma}{Lemma}[section]
\newtheorem{proposition}{Proposition}[section]

\theoremstyle{definition}
\newtheorem*{definition}{Definition}
\newtheorem{remark}{Remark}

\definecolor{myred}{RGB}{226,56,18}
\definecolor{myorange}{RGB}{228,139,0}
\definecolor{mygreen}{RGB}{4,215,17}
\definecolor{mygrey}{RGB}{180,180,180}

\def\Cx{\mathbb{C}}
\def\Chat{\widehat{\mathbb{C}}}
\def\intr{\mathrm{int}}

\begin{document}

\title{\textsc{\textbf{Escaping points of commuting meromorphic functions with finitely many poles}}}

\author{Gustavo R.~Ferreira\thanks{Email: \texttt{gustavo.rodrigues-ferreira@open.ac.uk}}\\
  \small{School of Mathematics and Statistics, The Open University}\\
  \small{Milton Keynes, MK7 6AA, UK}
}

\maketitle

\begin{abstract}
Let $f$ and $g$ be commuting meromorphic functions with finitely many poles. By studying the behaviour of Fatou components under this commuting relation, we prove that $f$ and $g$ have the same Julia set whenever $f$ and $g$ have no simply connected fast escaping wandering domains. By combining this with a recent result of Tsantaris', we obtain the strongest statement (to date) regarding the Julia sets of commuting meromorphic functions. In order to highlight the difference to the entire case, we show that transcendental meromorphic functions with finitely many poles have orbits that alternate between approaching a pole and escaping to infinity at strikingly fast rates.
\end{abstract}


\section{Introduction}
For any meromorphic function $f:X\to\Chat$, where $X\in\{\Cx, \Chat\}$, we define its \textit{Fatou set} as
\[ F(f) = \left\{z\in X : \text{$\{f^n\}_{n\in\mathbb{N}}$ is a normal family and $f^n(z)$ is defined for all $n\in\mathbb{N}$}\right\}. \]
The complement of the Fatou set is known as the \textit{Julia set}, denoted $J(f)$, and it exhibits fascinating topological, geometrical and dynamical properties. Ever since Fatou and Julia first proved that any rational function of degree $\geq 2$ has a non-empty Julia set in the 1920s, much effort has been made to understand the behaviour of these sets.

Apart from being an interesting subject in its own right, this subject of holomorphic dynamics can be used to prove statements in general complex analysis. For example, in 1990, Eremenko gave a dynamical proof of a theorem by Ritt, classifying pairs of commuting rational functions \cite{Ere90,Rit23}. Unlike Ritt's elaborate proof, Eremenko's reasoning pivoted about a simple fact: if two rational functions commute, they have the same Julia set.

Although this fact was proved by both Fatou and Julia in 1923, the corresponding problem for transcendental functions -- which could serve as a first step in classifying commuting analytic functions in a more general setting -- is still open. After Baker proved this fact for commuting entire functions without escaping Fatou components in \cite{Bak84}, much progress has been made for entire functions by considering the \textit{fast escaping set} introduced by Bergweiler and Hinkkanen in \cite{BH99}. This is the set
\[ A(f) = \left\{z\in\Cx : \text{there exists $l\in\mathbb{N}$ such that $|f^{n+l}(z)|\geq M(R, f^n)$ for all $n\in\mathbb{N}$}\right\}, \]
where $M(r, f)$ denotes the maximum modulus $M(r, f) = \max\{|f(z)| : |z| = r\}$ and $R$ is chosen so that $M(r, f) > r$ for all $r\geq R$. Bergweiler and Hinkkanen showed that, for commuting entire functions $f$ and $g$, $g^{-1}\left(A(f)\right) \subset A(f)$ \cite[~Theorem 5]{BH99} and used this to show that commuting implies $J(f) = J(g)$ whenever $A(f)\subset J(f)$ and $A(g)\subset J(g)$. More recently, Benini, Rippon and Stallard \cite{BRS16} showed that, for entire functions $f$ and $g$, commuting implies having the same Julia set except possibly when $f$ and $g$ have simply connected fast escaping wandering domains. The overall strategy used in each of the papers \cite{Bak84}, \cite{BH99} and \cite{BRS16} is to show that if $f$ and $g$ commute and the components of $F(f)$ and $F(g)$ are only of certain types, then $g\left(F(f)\right)\subset F(f)$. It then follows that $F(f)\subset F(g)$, by Montel's theorem, and similarly $F(g)\subset F(f)$.

Here, we concern ourselves with transcendental meromorphic functions. Following Osborne and Sixsmith \cite{OS16}, we say that $f$ and $g$ commute -- or are permutable -- if, for every $z\in\Cx$, either $f\left(g(z)\right) = g\left(f(z)\right)$ or neither $f\left(g(z)\right)$ nor $g\left(f(z)\right)$ are defined. An immediate consequence of these requirements is that if $f$ and $g$ commute then their poles are the same. Indeed, let $w$ be a pole of $f$. Then, $g\left(f(w)\right)$ is undefined, which implies that $f\left(g(w)\right)$ is undefined as well. However, as a meromorphic function $f:\Cx\to\Chat$, the only value for which $f$ is undefined is $\infty$, so $g(w) = \infty$ and $w$ is a pole of $g$.

A recent paper by Tsantaris includes the result \cite[~Theorem 1.6]{Tsa19} that $J(f) = J(g)$ whenever $f$ and $g$ commute, for all non-entire meromorphic functions $f$ and $g$ apart from those in the class $\mathcal{P}$ of meromorphic functions with a single pole which is also an omitted value. Interestingly, his proof used a completely different strategy, bypassing the behaviour of Fatou components entirely.

In this paper, we adopt a strategy resembling the ones used for transcendental entire functions to tackle commuting meromorphic functions with finitely many poles, including those in class $\mathcal{P}$, thus bringing transcendental meromorphic functions ``up to speed'' on the problem of commuting functions. More specifically, we will look at the way in which non-escaping and escaping orbits are affected by the commuting relation and so obtain new information about how the components of $F(f)$ and $F(g)$ are affected by this relation. First, we generalise Baker's methods to prove the following result about non-escaping orbits.

\begin{theorem} \label{thm:fin}
Let $f$ and $g$ be commuting meromorphic functions. Assume that, on a component $U$ of $F(f)$, the sequence $(f^n)_{n\geq 1}$ admits a subsequence $(f^{n_k})_{k\geq 1}$ that converges locally uniformly to a finite limit function. Then, $g(U)\subset F(f)$.
\end{theorem}

Next, we turn to Bergweiler and Hinkkanen's approach. Although their definition of the fast escaping set does not make sense when $f$ has poles, we can use an alternative definition of $A(f)$ formulated by Rippon and Stallard \cite{RS05}; see Section \ref{sec:FatouJulia}. With this in mind, we prove the following.

\begin{theorem} \label{thm:Bf}
Let $f$ and $g$ be commuting meromorphic functions with finitely many poles. Let $z\in\Cx$ be a point that is escaping but not fast escaping under $f$. Then, $g(z)\notin A(f)$.
\end{theorem}

As noted above, Bergweiler and Hinkkanen showed that for commuting transcendental entire functions $f$ and $g$ we have $g\left(A(f)^c\right)\subset A(f)^c$, which is a stronger statement than Theorem \ref{thm:Bf}. Later in the introduction we discuss why our proof of Theorem \ref{thm:Bf} requires the extra hypothesis that the point $z$ is escaping.

In keeping with our aim of studying the behaviour of Fatou components under the commuting relation, we also prove a meromorphic version of Benini, Rippon and Stallard's result \cite[~Proposition 3.3]{BRS16}. Recall that a \textit{Baker wandering domain} is a multiply connected wandering domain $U\subset F(f)$ such that every $f^n(U)$ is bounded, $\mathrm{dist}\left(0, f^n(U)\right) \to +\infty$ as $n\to +\infty$ and $f^n(U)$ surrounds the origin for every sufficiently large $n$.

\begin{theorem} \label{thm:BRS}
Let $f$ and $g$ be commuting transcendental meromorphic functions with finitely many poles, and let $U\subset F(f)$ be a Baker wandering domain. Then, $g(U)$ is a Baker wandering domain of f.
\end{theorem}

It was proved by Baker \cite{Bak87} that functions in class $\mathcal{P}$ do not have Baker wandering domains, so the equality of the Julia sets in the context of Theorem \ref{thm:BRS} is already established by Tsantaris' result. Rather, Theorem \ref{thm:BRS} is about the fact that Baker wandering domains of $f$ are -- in a sense -- ``preserved'' by $g$.

After these results on the behaviour of escaping points, we address the problem of equality of the Julia sets for general meromorphic functions. First, Theorems \ref{thm:fin} and \ref{thm:Bf} combine to give the following.

\begin{theorem} \label{thm:main}
Let $f$ and $g$ be transcendental meromorphic functions with finitely many poles, and suppose that $A(f)\subset J(f)$ and $A(g)\subset J(g)$. Then, if $f$ and $g$ commute, $J(f) = J(g)$.
\end{theorem}

We prove the above four theorems in Section \ref{sec:FatouJulia}. Next, combining Theorem \ref{thm:main} with Tsantaris' result yields our strongest statement on the equality of Julia sets.

\begin{theorem} \label{thm:Main}
Let $f$ and $g$ be commuting non-entire transcendental meromorphic functions. Then, $J(f) = J(g)$ except possibly when $f$ and $g$ are in class $\mathcal{P}$ and have simply connected fast escaping wandering domains.
\end{theorem}

To see how Theorem \ref{thm:Main} follows from Tsantaris' result and Theorem \ref{thm:main}, we observe that if $f$ and $g$ are in class $\mathcal{P}$ and have no simply connected fast escaping wandering domains, then they have no fast escaping Fatou components. Indeed, any fast escaping Fatou component must be a wandering domain \cite[~Theorem 2]{RS05} and, as shown by Baker \cite{Bak87}, functions in class $\mathcal{P}$ have no multiply connected wandering domains.

It is an interesting open question whether functions in class $\mathcal{P}$ can have simply connected fast escaping wandering domains.

Now, we return to the point made earlier that Theorem \ref{thm:Bf} is not as strong as the corresponding result Bergweiler and Hinkkanen have for entire functions, since Theorem \ref{thm:Bf} requires $z$ to be escaping. Nor is Theorem \ref{thm:fin} as strong as it could be for entire functions. Indeed, although Baker does not remark on this, his method implies that if $f$ and $g$ are commuting transcendental entire functions then $g$ maps non-escaping $f$-orbits to non-escaping $f$-orbits. This property, however, does not follow immediately in the meromorphic case. The reason for this is that $f$ can have non-escaping orbits with a subsequence converging to a pole. Since $f$ and $g$ share poles upon commuting, this means that $g$ could potentially map this non-escaping orbit to an escaping one. In order to give a precise description of this phenomenon, we introduce the following new concept.

\begin{definition}
Let $f$ be a transcendental meromorphic function with finitely many poles. A point $z\in\Cx$ is said to have a \textit{ping-pong orbit} if there exist a pole $p\in\Cx$ of $f$, subsequences $(f^{m_k})_{k\geq 1}$ and $(f^{n_k})_{k\geq 1}$ and $M \geq 1$ such that
\begin{enumerate}[(i)]
    \item $f^{m_k}(z)\to p$ and $f^{n_k}(z)\to \infty$ as $k\to +\infty$;
    \item for all $k\in\mathbb{N}$, $|m_k - n_k| \leq M$ and $|n_k - m_{k+1}| \leq M$.
\end{enumerate}
We denote the set of points with ping-pong orbits by $BU_P(f)$, and study it in Section \ref{sec:PingPong}.
\end{definition}

We point out that $BU_P(f)$ is a subset of the so-called \textit{bungee set} $BU(f)$; defined by Osborne and Sixsmith \cite{OS16b}, this is the set of points whose orbit is neither escaping nor bounded. Clearly, the orbit of any point in $BU(f)$ has a subsequence escaping to infinity, and another with a finite limit. In the case of a ping-pong orbit, we ask that this limit be a pole; indeed, condition (ii) of the definition implies that a pole must be involved somehow.

Up until this point, our discussion about $BU_P(f)$ has been purely hypothetical; there is, \textit{a priori}, no reason to assume that such points actually exist. In Subsection~\ref{ssec:ex}, we construct meromorphic functions with a single pole and ping-pong wandering domains. This shows that we are justified in focusing on escaping points in Theorem \ref{thm:Bf}.

Finally, we show that the presence of ping-pong orbits is not an oddity of the examples constructed here. The bungee set, for instance, is ubiquitous: Osborne and Sixsmith showed in \cite{OS16b} that $BU(f)$ is always dense in the Julia set of transcendental entire functions. This result was extended to quasiregular functions by Nicks and Sixsmith \cite{NS18}, and then to quasimeromorphic ones by Warren \cite{War20}. Our final theorem, to be proved in Section 3, shows the same for $BU_P(f)$, and adds something special about its rate of escape.

\begin{theorem} \label{thm:pingpong}
Let $f$ be a non-entire transcendental meromorphic function with a finite number of poles. Then, there exists a dense set of points in $J(f)$ with ping-pong orbits. Furthermore, if $f$ is not in class $\mathcal{P}$, the escaping subsequence of their orbits can escape arbitrarily fast.
\end{theorem}

\textsc{Acknowledgements.} I am deeply grateful to my supervisors, Phil Rippon and Gwyneth Stallard, for their contributions to this work.

\section{Fatou and Julia sets of commuting meromorphic functions} \label{sec:FatouJulia}
In this section, we prove Theorems \ref{thm:fin} to \ref{thm:main}.

\subsection{The non-escaping case}
In this subsection, we prove Theorem \ref{thm:fin}, in which we consider a non-escaping Fatou component of $f$. To do this, we use the following version of the blowing-up property of the Julia set, which is a standard result for entire and rational functions -- see, for instance, \cite[~Proposition 2.4.5]{MNTU}. Although it seems to be well-known for transcendental meromorphic functions as well, we were not able to find a proof for this version in the literature and include one here for completeness. First, we recall that the exceptional set $E(f) = \{z\in\Chat : \text{$\mathcal{O}^-(z)$ is finite}\}$ has at most two points, by Picard's theorem.

\begin{proposition} \label{prop:blowup}
Let $f$ be a transcendental meromorphic function and let $K\subset \Cx\setminus E(f)$ be a compact set. For any $z\in J(f)$ and any neighbourhood $V$ of $z$, there exists $N\in\mathbb{N}$ such that, for all $n\geq N$,
\[ f^n(V)\supset K. \]
\end{proposition}
\begin{proof}
We divide the proof in two cases based on whether
\[ \mathcal{O}^-(\infty) = \{w\in\Cx : \text{there exists $n\in\mathbb{N}$ such that $f^n(w) = \infty$}\} \]
is finite or infinite. Without loss of generality, we assume that $f$ is non-entire.

In the finite case, it follows from Picard's theorem that $f$ is in class $\mathcal{P}$. Without loss of generality, we assume that its pole is at the origin. Now, if $z = 0$, any neighbourhood $V$ of $z$ is mapped by $f$ onto a neighbourhood of $\infty$. By Picard's theorem, $f^n(V)$, $n\geq 2$, contains all of $\Chat$ apart from the exceptional points of $f$, and so in particular $f^n(V)\supset K$ for any compact set $K\subset \Cx\setminus E(f)$. If $z\neq 0$, but $V$ contains $0$, the same argument applies; if, on the other hand, $V$ does not contain the origin, then $f^n(V)$ does not contain a pole of $f$ for any $n\in\mathbb{N}$. This means that the usual argument for the blowing-up property for entire functions applies; we refer to \cite[~Proposition 2.4.5]{MNTU}.

However, if $\mathcal{O}^-(\infty)$ is infinite, then it follows from Montel's theorem that $J(f) = \overline{\mathcal{O}^-(\infty)}$; see, for instance, \cite[~p. 4]{Ber93}. In particular, any neighbourhood $V$ of a point $z\in J(f)$ contains a point $w$ such that, for some $N_0\in\mathbb{N}$, $f^{N_0}(w) = \infty$. This means that $V$ contains a small disc $D\ni w$ such that $f^{N_0}(D)$ is a neighbourhood of infinity and, again by Picard's theorem, $f^n(D)$ omits at most two points on the Riemann sphere for $n\geq N_0 + 1$; namely, the exceptional values of $f$. In particular, $f^n(D)$ contains any compact set $K\subset \Cx\setminus E(f)$ for $n\geq N_0 + 1$.
\end{proof}

We will also need the following property of meromorphic functions.
\begin{lemma} \label{lem:Lip}
Let $g:\Cx\to\Chat$ be a transcendental meromorphic function and $K$ a compact subset of $\Cx$. Then, $g|_K$ is Lipschitz continuous with respect to the Euclidean and spherical metrics (on $\Cx$ and $\Chat$, respectively).
\end{lemma}
\begin{proof}
As a meromorphic function, $g$ has at most finitely many poles $w_1, w_2, \ldots, w_m$ in $K$. Also, we note that we can, if necessary, slightly enlarge $K$ so that $\partial K$ does not contain any poles of $g$. Around each pole, we have a neighbourhood $V_j$ of $w_j$ and a positive integer $m_j$, the order of $w_j$ as a pole of $g$, such that
\begin{equation} \label{eq:g}
    g(z) = \frac{h_j(z)}{(z - w_j)^{m_j}} \quad \text{for every $z\in \overline{V_j}$},
\end{equation}
where $h_j$ is holomorphic in a neighbourhood of $\overline{V_j}$ with $h_j(w_j) \neq 0$ and $\overline{V_j}\subset K$. Now, we write $K$ as
\[ K = \overline{V_1}\cup\overline{V_2}\cup\cdots\cup\overline{V_m}\cup\left(K\setminus\bigcup_{j=1}^m V_j\right). \]
Since $K\setminus\bigcup_{j=1}^m V_j$ is a compact set without poles of $g$, the ``Euclidean-to-spherical'' derivative
\[ g^\sharp(z) = \frac{|g'(z)|}{1 + |g(z)|^2} \]
is well-defined and continuous, and thus bounded above on this set. Now, on each $\overline{V_j}$, we use (\ref{eq:g}) to obtain
\[ g^\sharp(z) = \frac{|z - w_j|^{m_j-1}\cdot|(z - w_j)h_j'(z) - m_jh_j(z)|}{|z - w_j|^{2m_j} + |h_j(z)|^2}, \]
which, since $h_j(w_j) \neq 0$, is also well-defined and continuous -- and thus bounded above. It follows that $g^\sharp$ is bounded above on $K$, whence $g|_K$ is Lipschitz continuous as promised.
\end{proof}

We are now ready to prove Theorem \ref{thm:fin}.

\begin{proof}[Proof of Theorem \ref{thm:fin}]
Take $z\in U$ and a neighbourhood $V$ of $z$ with $\overline{V}\subset U$. Since $(f^{n_k})_{k\geq 1}$ converges locally uniformly to a finite limit in $\overline{V}$, it follows that all $f^{n_k}(V)$ are contained in a single compact set $K\subset \Cx$. By Lemma \ref{lem:Lip}, $g|_K$ is Lipschitz continuous with some Lipschitz constant $M > 0$ for the Euclidean and spherical metrics. Now, take $\epsilon > 0$. If necessary, we pass to a smaller neighbourhood $V'\subset V$ such that
\[ \mathrm{diam}\left(f^{n_k}(V')\right) < \epsilon/M \quad\text{for every $k\geq 1$}, \]
where the diameter is taken with respect to the spherical metric -- notice that $V'$ is guaranteed to exist by the Arzelà-Ascoli theorem. Then, by using the commuting hypothesis and the fact that $g$ is Lipschitz continuous on $K$, we obtain that, for every $k\geq 1$,
\[ \mathrm{diam}\left(f^{n_k}\left(g(V')\right)\right) = \mathrm{diam}\left(g\left(f^{n_k}(V')\right)\right) \leq M\mathrm{diam}\left(f^{n_k}(V')\right) < \epsilon. \]
This implies that $f^{n_k}\left(g(V')\right)$ will never contain any compact set $K'\subset \Cx$ with diameter greater than $\epsilon$. Thus, by Proposition \ref{prop:blowup}, $g(z)$ cannot lie in $J(f)$.
\end{proof}

An immediate consequence of Theorem \ref{thm:fin} is that, whenever it can be assured that neither $f$ nor $g$ have escaping points in the Fatou set, then commuting implies having the same Julia set. This is the case if, for instance, $f$ and $g$ have finitely many singular values; see \cite[~p. 652]{BKL92} and \cite[~p. 3252]{RS99}. However, this also follows immediately from Theorem \ref{thm:Main}, and so we refrain from dwelling on it.

\subsection{The escaping case}
Much of this subsection is devoted to proving Theorem \ref{thm:Bf}. We begin by explaining Rippon and Stallard's definition \cite{RS05} of the fast escaping set for meromorphic functions with finitely many poles.

Let $f$ be a transcendental meromorphic function with at most finitely many poles. First, for any Jordan curve $\gamma\subset\Cx$, we define its \textit{outer set} to be the closure of its unbounded complementary component. A sequence of Jordan curves $\gamma_n$ with associated outer sets $E_n$, $n\in\mathbb{N}$, is called an \textit{outer sequence} for $f$ if it satisfies the following properties.
\begin{enumerate}[(i)]
    \item There exists $r > 0$ such that all poles of $f$ are contained in $\{z\in\Cx : |z| < r\}$ and all the curves $\gamma_n$ surround $\{z\in\Cx : |z| = r\}$;
    \item $\mathrm{dist}(0, \gamma_n)\to + \infty$ as $n\to +\infty$;
    \item For all $n\geq 1$, $\gamma_{n+1}\subset f(\gamma_n)$;
    \item For all $n\geq 1$, any component of $f^{-1}(E_{n+1})$ lies either in $E_n$ or in $\{z\in\Cx : |z| < r\}$.
\end{enumerate}
Rippon and Stallard showed that every transcendental meromorphic function with at most finitely many poles admits an outer sequence, and used this fact to define the fast escaping set as
\[ A(f) = \left\{z\in\Cx : \text{there exists $l\in\mathbb{N}$ such that $f^{n+l}(z)\in E_n$ for all $n\in\mathbb{N}$}\right\}, \]
where $E_n$ is an outer sequence for $f$. Originally, this set was denoted $B(f)$, but the definition above is equivalent to Bergweiler and Hinkkanen's if $f$ is entire (see \cite{RS12}). Hence, we denote it by $A(f)$ without fear of ambiguity. We also have the following properties of $A(f)$, proved in \cite{RS05}.

\begin{lemma} \label{lem:RS}
For any transcendental meromorphic function $f$ with finitely many poles, the following hold.
\begin{enumerate}[(i)]
    \item $A(f)$ is independent of the choice of outer sequence.
    \item $A(f)$ is non-empty and intersects the Julia set.
    \item If $U$ is a Fatou component that intersects $A(f)$, then $U\subset A(f)$. Furthermore, $U$ is a wandering domain.
    \item $J(f) = \partial A(f)$, and if $f$ does not have wandering domains, then $J(f) = \overline{A(f)}$.
\end{enumerate}
\end{lemma}

Next, we set up some preliminary results. The first one is a version by Domínguez \cite{Dom98} of a theorem of Bohr's, and the second is a consequence of the open mapping theorem.

\begin{lemma} \label{lem:Dom}
Let $f$ be analytic on $\{z\in\Cx : |z| > R\}$ for some $R > 0$ and such that $M(r, f) \to +\infty$ as $r\to +\infty$. Then, for any large enough $\rho > R$, $f\left(\{z\in\Cx : R < |z| < \rho\}\right)$ contains a circle of radius $r' > cM(\rho/2, f)$, where $c > 0$ is an absolute constant.
\end{lemma}

\begin{lemma} \label{lem:BIIB}
Let $h:\Cx\to\Chat$ be a meromorphic function. Then, for any bounded open set $D$ such that $\overline{D}$ contains no poles of $h$, we have $\partial h(D)\subset h(\partial D)$.
\end{lemma}

With these at hand, we can prove the following lemma -- which will be a central tool in our proof of Theorem \ref{thm:Bf}.

First, it is important to remark that, whenever we talk about the outer boundary (or outer boundary component) of a bounded set $D$, we mean the boundary of its unbounded complementary component.

\begin{lemma} \label{lem:outseqs}
Let $f$ and $g$ be commuting transcendental meromorphic functions with finitely many poles. Then, $f$ has an outer sequence $(F_n)_{n\in\mathbb{N}}$ with Jordan curves $\gamma_n$ such that the outer boundaries of $g(\gamma_n)$ are Jordan curves $\Gamma_n$ whose outer sets form an outer sequence $(E_n)_{n\in\mathbb{N}}$ for $f$.
\end{lemma}
\begin{proof}
We start by choosing $r_0 > 0$ such that $\Delta = \{z\in\Cx : |z| \leq r_0\}$ contains all poles of both $f$ and $g$ (with none on the boundary), $M(r_0, f) > r_0$ and $M(r_0, g) > r_0$. Next, take $R_1 > r_0$ large enough that Lemma \ref{lem:Dom} applies, $cM(R_1/2, g) > \max\{M(r_0, f), M(r_0, g)\}$ and $cM(r/2, f) > r$ for all $r\geq R_1$, where $c$ is the constant given by Lemma \ref{lem:Dom}. We define an increasing sequence of radii by
\[ R_{n+1} = cM(R_n/2, f) > R_n, \]
whence it follows that $R_n\to +\infty$ as $n\to +\infty$. Notice that since $f$ and $g$ are holomorphic in $\Cx\setminus\Delta$,
\begin{equation} \label{eq:ineqPhil}
    M(r_0, f) < cM(R_n/2, g) \quad\text{and}\quad M(r_0, g) < cM(R_n/2, g) \quad\text{for all $n\in\mathbb{N}$,}
\end{equation}
by the maximum modulus principle applied to the annuli $A(r_0, R_{n+1}/2)\supset A(r_0, R_n/2)$. Pick a Jordan curve $\gamma_1$ surrounding $\{z\in\Cx : |z| = R_1\}$, and inductively define $\gamma_{n+1}$ as the outer boundary component of $f\left(\intr(\gamma_n)\setminus\Delta\right)$; here and throughout, $\intr(\gamma)$ denotes the bounded complementary component of the Jordan curve $\gamma$. Notice that $f$ maps domains onto domains by the open mapping theorem, and hence $\gamma_{n+1}$ is a Jordan curve. Furthermore, by applying Lemma \ref{lem:Dom} to the topological annuli $\intr(\gamma_n)\setminus\Delta$, $n\in\mathbb{N}$, we see that
\begin{equation} \label{eq:gamma}
    \text{$\gamma_n$ surrounds the circle $\{z\in\Cx : |z| = R_n\}$ for every $n\in\mathbb{N}$.}
\end{equation}
Finally, we set $F_n$ as the outer set for $\gamma_n$ for all $n\in\mathbb{N}$, and see immediately that $F_n$ defines an outer sequence for $f$. Indeed, properties (i) and (ii) follow from (\ref{eq:gamma}) and our choice of the sequence $R_n$; property (iii) follows from Lemma~\ref{lem:BIIB} and the definition of $\gamma_{n+1}$, and property (iv) is a consequence of the fact that $f\left(\intr(\gamma_n)\setminus\Delta\right)\subset \intr(\gamma_{n+1})$.

Now, define $\Gamma_n$ as the outer boundary component of $g\left(\intr(\gamma_n)\setminus\Delta\right)$ for all $n\in\mathbb{N}$. As was the case with $\gamma_{n+1}$ and $\gamma_n$, $\Gamma_n$ is a Jordan curve and is contained in $g(\gamma_n)$. Then, by (\ref{eq:ineqPhil}), (\ref{eq:gamma}) and Lemma \ref{lem:Dom}, $g\left(\intr(\gamma_n)\setminus\Delta\right)$ contains $\{z\in\Cx: |z| = cM(R_n/2, g)\}$, so, by Lemma \ref{lem:BIIB}, $\Gamma_n$ is a Jordan curve surrounding this circle. We want to show that the outer sets $E_n$ associated to $\Gamma_n$ form an outer sequence for $f$.

Properties (i) and (ii) are immediate consequences of the fact that $\Gamma_n$ surrounds $\{z\in\Cx : |z| = cM(R_n/2, g)\}$, since $R_n\to+\infty$ and $cM(R_n/2, g) > M(r_0, g) > r_0$. Proving that $E_n$ satisfies property (iii), that is, that $\Gamma_{n+1}\subset f(\Gamma_n)$, is a considerably more delicate task. First, since $(\gamma_n)_{n\in\mathbb{N}}$ is itself an outer sequence, we have
\[ \Gamma_{n+1}\subset g(\gamma_{n+1})\subset g\left(f(\gamma_n)\right) = f\left(g(\gamma_n)\right), \]
where the last equality uses the fact that $f$ and $g$ commute. However, since $\Gamma_n\subset g(\gamma_n)$ and not the opposite, we are not done. In particular, we must take care of any points on $\gamma_n$ that are mapped into $\Delta$ by $g$, since those can be influenced by the poles of $f$. In order to control this, we shall make small modifications to the annuli used to define $\gamma_{n+1}$ and $\Gamma_n$, and then convince ourselves that these modifications were not important.

For any $n\geq 1$, denote by $A_n$ the topological annulus $\intr(\gamma_n)\setminus\Delta$, and by $\Omega_n$ the set obtained by removing $g^{-1}(\Delta)$ and $f^{-1}(\Delta)$ from $A_n$ (that is the grey-shaded region in Figure \ref{fig:iii}). Then, since $g(\Omega_n)$ omits $\Delta$ by the definition of $\Omega_n$, and thus omits the poles of $f$, we see that $f\circ g$ is analytic in $\Omega_n$. By the commuting hypothesis, $g\circ f$ is also analytic and equals $f\circ g$ throughout $\Omega_n$.

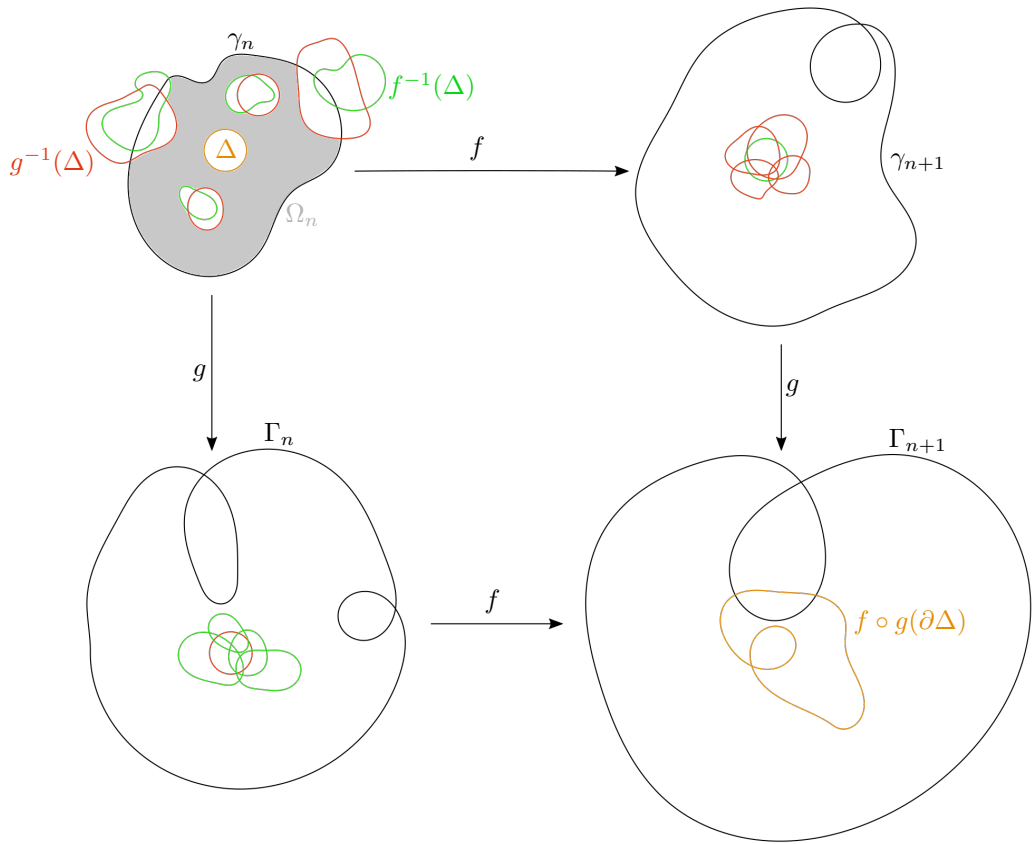
\begin{figure}[!h]
    \centering
    \begin{tikzpicture}
        \node[anchor=south west,inner sep=0] at (0,0) {\includegraphics[width=\textwidth]{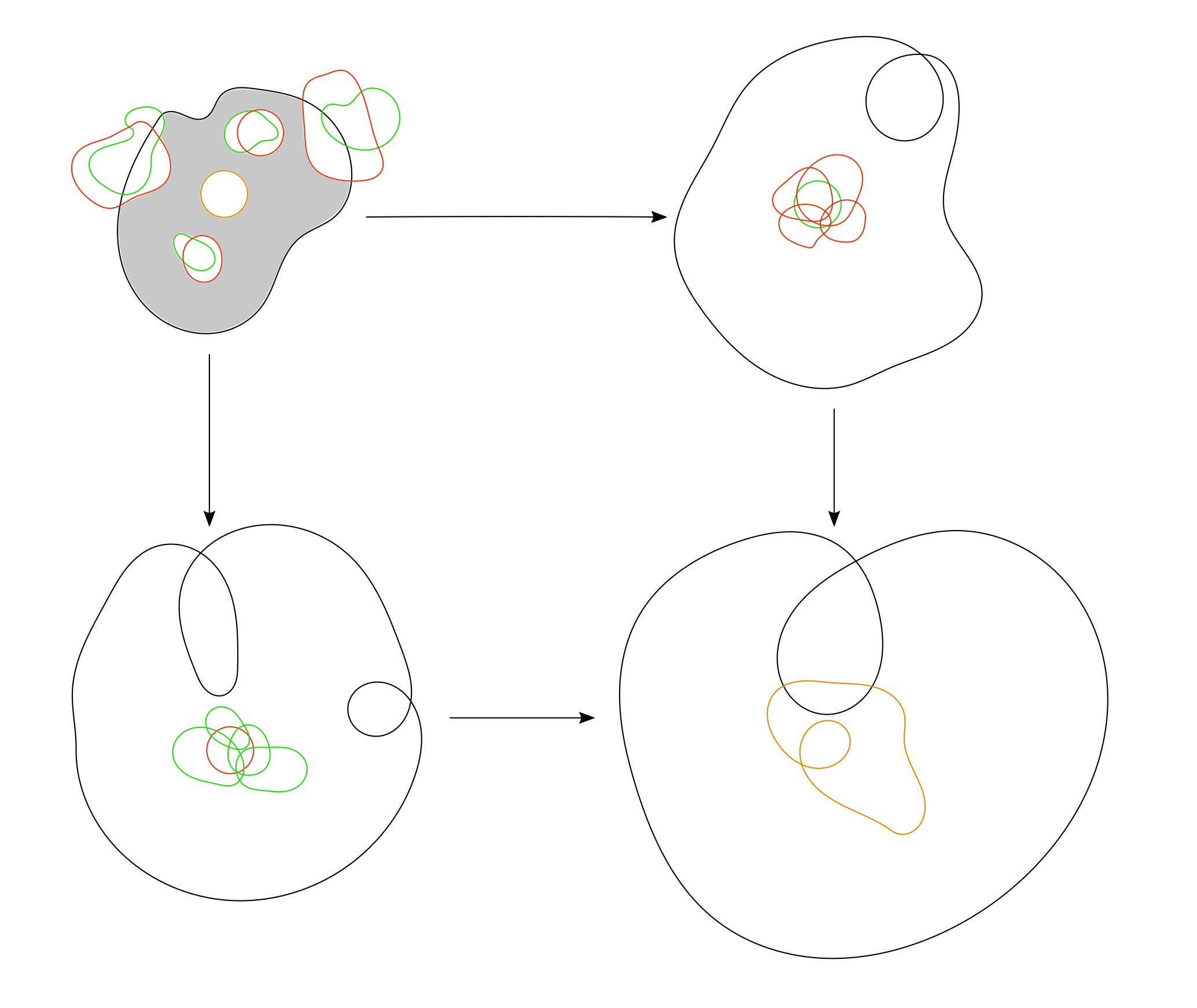}};
        \node[above] at (6,9.5) {$f$};
        \node[above] at (6.2,3.5) {$f$};
        \node at (2.35,6.8) {$g$};
        \node at (10.15,6.6) {$g$};
        \node at (2.9,11.2) {$\gamma_n$};
        \node[myred,anchor=east] at (1.1,9.6) {$g^{-1}(\Delta)$};
        \node[mygreen,anchor=west] at (4.7,10.6) {$f^{-1}(\Delta)$};
        \node[myorange] at (2.7,9.78) {$\Delta$};
        \node[mygrey] at (3.7,8.9) {$\Omega_n$};
        \node at (3.4,6) {$\Gamma_n$};
        \node at (11.8,5.9) {$\Gamma_{n+1}$};
        \node at (11.8,9.6) {$\gamma_{n+1}$};
        \node[myorange] at (11.7,3.5) {$f\circ g(\partial\Delta)$};
    \end{tikzpicture}
    \caption{The relevant sets and curves in the proof of Lemma \ref{lem:outseqs}.}
    \label{fig:iii}
\end{figure}

Notice that the outer boundary of $\Omega_n$ consists of pieces of $\gamma_n$, $\partial g^{-1}(\Delta)$ and $\partial f^{-1}(\Delta)$; in Figure \ref{fig:iii}, these are black, red and green pieces, respectively. We shall show that the outer boundary of $f\left(g(\Omega_n)\right) = g\left(f(\Omega_n)\right)$ consists solely of images of black pieces.

First, the outer boundary of $f(\Omega_n)$ contains, by Lemma \ref{lem:BIIB}, only parts of $\gamma_{n+1}$ and $f\left(\partial g^{-1}(\Delta)\right)$, since any pieces of $f\left(\partial f^{-1}(\Delta)\right)$ are mapped into $\partial\Delta$ by definition. Mapping this forward under $g$, we see that all green pieces are ``trapped'' inside a circle of radius $M(r_0, g) < cM(R_{n+1}/2, g)$ by (\ref{eq:ineqPhil}). By an analogous argument considering the images of red pieces under $g\circ f$, we see that these are trapped inside a circle of radius $M(r_0, f) < cM(R_{n+1}/2, g)$.

It follows that any boundary point of $f\left(g(\Omega_n)\right)$ outside a circle of radius $cM(R_{n+1}/2, g)$ is contained in $f\left(g(\gamma_n)\right)$; in particular, $f\left(g(\Omega_n)\right) = g\left(f(\Omega_n)\right)$ contains the circle of radius $cM(R_{n+1}/2, g)$ and its outer boundary is $\Gamma_{n+1}$. Finally, since we know (by the commuting hypothesis and (\ref{eq:ineqPhil})) that $f\left(g\left(\partial f^{-1}(\Delta)\right)\right)$ is surrounded by $\{z\in\Cx : |z| = cM(R_{n+1}/2, g)\}$, the outer boundary of $f\left(g(\Omega_n)\right)$ equals that of $f\left(\intr(\Gamma_n)\setminus\Delta\right)$, which is contained in $f(\Gamma_n)$ by Lemma \ref{lem:BIIB}, and thus $\Gamma_{n+1}\subset f(\Gamma_n)$.

Finally, property (iv) follows from (iii). Indeed, consider the set $D = \intr(\Gamma_n)\setminus\Delta$. We know, by continuity and openness of $f$, that the outer boundary of $f(D)$ is a Jordan curve, which -- by (iii) -- must be equal to $\Gamma_{n+1}$. Consequently, no point in $D$ is mapped to $E_{n+1} = \Cx\setminus\intr(\Gamma_{n+1})$ by $f$. We conclude that $(E_n)_{n\in\mathbb{N}}$ is an outer sequence for $f$.
\end{proof}

We use our hard-won outer sequence to complete the proof of Theorem \ref{thm:Bf} as follows.

\begin{proof}[Proof of Theorem \ref{thm:Bf}]
Since the fast escaping set is not affected by our choice of outer sequence, we pick $(F_n)_{n\in\mathbb{N}}$ and $(E_n)_{n\in\mathbb{N}}$ to be the outer sequences given by Lemma \ref{lem:outseqs}. We assume now that $z$ is an escaping point of $f$ and $g(z)\in A(f)$, and show that this leads to a contradiction. By the definition of $A(f)$, there exists $l\in\mathbb{N}$ such that $f^{n+l}\left(g(z)\right)\in E_n$ for all $n\in\mathbb{N}$, and by the commuting hypothesis this is equivalent to $g\left(f^{n+l}(z)\right)\in E_n$, which in turn is the same as $f^{n+l}(z)\in g^{-1}(E_n)$ for all $n\in\mathbb{N}$.

We claim that our choice of outer sequences implies that, for every $n\in\mathbb{N}$, a component $g^{-1}(E_n)$ is either a neighbourhood of a pole of $g$ or is contained in $F_n$. Indeed, if $\Gamma_n$ and $\gamma_n$ are the Jordan curves associated with $E_n$ and $F_n$ respectively, then by Lemma \ref{lem:outseqs} we have that $g\left(\intr(\gamma_n)\setminus\Delta\right)\subset \intr(\Gamma_n)$. As $E_n = \Cx\setminus\intr(\Gamma_n)$ and $F_n = \Cx\setminus\intr(\gamma_n)$ by definition, our claim follows. Since our original point $z$ is escaping and all poles of $g$ are contained in the bounded set $\Delta$, it follows that $f^{n+l}(z)\in g^{-1}(E_n)$ means, in fact, that $f^{n+l}(z)\in F_n$, which in turn implies $z\in A(f)$; we have obtained our contradiction.
\end{proof}

\begin{remark}
A small modification of the proof of Lemma \ref{lem:outseqs} yields a similar result to Theorem \ref{thm:Bf} in the case where $f$ and $g$ are transcendental meromorphic functions with finitely many poles such that $h\circ f = g\circ h$ for some entire function $h$. In this case, we deduce that $h^{-1}\left(A(g)\right)\subset A(f)$, thus generalising \cite[~Theorem 5]{BH99}. To prove this, we define $\gamma_n$ to be an outer sequence for $f$ and $\Gamma_n\subset h(\gamma_n)$ to be an outer sequence for $g$; the only necessary modification is to take $\Omega_n = A_n\setminus h^{-1}(\Delta)$. A particularly interesting application of this is the case where $g$ is in class $\mathcal{P}$, $f$ is entire and $h = \exp$, so that it yields that $\exp^{-1}\left(A(g)\right) \subset A(f)$.
\end{remark}

Finally, we are ready to prove Theorem \ref{thm:main}.

\begin{proof}[Proof of Theorem \ref{thm:main}]
We are going to prove that $g\left(F(f)\right)\subset F(f)$, which implies (as indicated in the introduction) that $F(f)\subset F(g)$ and hence, by symmetry, that $F(f) = F(g)$.

Suppose, then, there exists some $z$ in a component $U\subset F(f)$ such that $g(z)\in J(f)$. If $U$ is not an escaping component, then Theorem \ref{thm:fin} applies and we get a contradiction. If it is escaping, it cannot be fast escaping since $A(f)\subset J(f)$; we will apply Theorem \ref{thm:Bf} and obtain a contradiction.

To this end, take a neighbourhood $V\subset U$ of $z$. By the open mapping theorem, $g(V)$ is a neighbourhood of $g(z)\in J(f)$, and so by Lemma \ref{lem:RS}(iv) there exists some $w'\in g(V)$ such that $w'\in A(f)$. But $w'$ is the image under $g$ of some $z'\in V\subset U$, and since $U$ is not fast escaping this contradicts Theorem \ref{thm:Bf}; we are done.
\end{proof}

We end this section with a proof of Theorem \ref{thm:BRS}, after recalling a result by Rippon and Stallard \cite{RS08}.

\begin{lemma} \label{lem:RSBaker}
Suppose $f$ is a transcendental meromorphic function with finitely many poles and $U$ is a multiply connected wandering domain of $f$. Then, $U$ is a Baker wandering domain if and only if infinitely many of the components $U_n\supset f^n(U)$, $n = 0, 1, \ldots$, are multiply connected.
\end{lemma}

\begin{proof}[Proof of Theorem \ref{thm:BRS}]
We can assume without loss of generality that $f$ has at least one pole. Since $f$ has a Baker wandering domain, we know that $f\notin\mathcal{P}$ by \cite{Bak87}, so $J(f) = J(g)$ by \cite[~Theorem 1.6]{Tsa19}. Therefore, $U$ is a bounded component of $F(g)$, whence $V = g(U)$ is a Fatou component of $g$ (and thus of $f$) by a result of Herring \cite[~Theorem 2]{Her98}. Suppose now that $V$ is not a Baker wandering domain. Then, by Lemma \ref{lem:RSBaker}, there must be $N\in\mathbb{N}$ such that $f^n(V)$ is simply connected for all $n\geq N$. Now, Zheng \cite{Zhe06} proved that, for all sufficiently large $n$, $f^n(U)$ contains an annulus $A(r_n, s_n) := \{z\in\Cx : r_n < |z| < s_n\}$ such that $r_n\to \infty$ and $(2\pi)^{-1}\log(s_n/r_n) \to +\infty$ as $n\to +\infty$. To complete the argument, we reproduce the relevant parts of the proof of Proposition 3.3 in \cite{BRS16}.

Since $f^n\left(g(U)\right) = f^n(V)$, it follows by the commuting hypothesis that $g\left(f^n(U)\right) = f^n(V)$ as well. Now, for sufficiently large $n$, we can take $r\in (r_n, s_n)$ such that $\gamma = g\left(\{z\in\Cx : |z| = r\}\right)$ winds at least once around the origin (by Picard's theorem combined with the argument principle), contains points of modulus $M(r, g) > r$, and lies in $f^n(V)$ (because $\{z\in\Cx : |z| = r\}\subset A(r_n, s_n)\subset f^n(U)$). Since we assumed $f^n(V)$ to be simply connected, it must meet $A(r_n, s_n)\subset f^n(U)$, and hence $U_n$; this is a contradiction, as $U_n$ is a Baker wandering domain and thus multiply connected.
\end{proof}

\section{Existence of ping-pong orbits} \label{sec:PingPong}

In this section, we show the existence of ping-pong orbits. First, in Subsection \ref{ssec:ex}, we construct examples of ping-pong wandering domains.

\subsection{Constructing wandering domains with a ping-pong orbit} \label{ssec:ex}
We shall modify a method used by Martí-Pete in the proof of \cite[~Theorem 1.1]{Mar19}. We choose $f$ to be of the form $f(z) = g(z) + 1/z$ with $g$ entire, and focus on obtaining $g$ through appropriate approximations. We will construct two families of discs, $A_m$ and $B_m$, $m\in\mathbb{N}$, which will give rise (respectively) to the escaping part of a ping-pong orbit for a wandering domain $U$ and pre-images of an attracting domain $V$.

Let $R > 0$, and $k_m$ be a sequence of positive real numbers such that, for all $m\in\mathbb{N}$, $k_m > 5/2$ and $k_{m+1} > k_m + 3R$. Let
\[ \text{$A_m = B(k_m, R)$ and $B_m = B\left(\frac{k_{m+1} + k_m}{2}, \frac{R}{4}\right)$ for all $m\in\mathbb{N}$,} \]
$B_+ = \overline{B(2, 1/4)}$ and $B_- = \overline{B(-5, 1)}$; here and throughout, $B(z, r)$ denotes the open disc of centre $z$ and radius $r > 0$. Notice that, for all $m\in\mathbb{N}$, the sets $1/A_m$ and $1/B_m$ are also discs; they are contained in $\mathbb{D}$ and converge to $0$ as $m\to +\infty$. Now, define
\[ F := \overline{\mathbb{D}}\cup B_-\cup B_+\cup \bigcup_{m\in\mathbb{N}} (\overline{A_m}\cup \overline{B_m}). \]
Since $F$ is a countable union of hand-picked disjoint compact sets, it satisfies the hypotheses of Nersesyan's theorem (see \cite[~Chapter IV]{Gai87}): $\Chat\setminus F$ is connected and locally connected at infinity, and all components of the interior of $F$ are bounded. It follows that there exists an entire function $g$ satisfying the following inequalities.
\begin{equation} \label{eq:DavidIneq}
 \begin{cases}
        |g(z) + 1/z - 1/k_{m+1}| < \epsilon_m, & z\in A_m\\
        |g(z) + 1/z - 2| < 1/5, & z\in B_+\cup\left(\bigcup_{m\in\mathbb{N}} B_m\right)\\
        |g(z)| < R/4, & z\in\mathbb{D}\\
        |g(z) + 1/z - (z + 5)| < 1/2, & z\in B_-
    \end{cases},
\end{equation}
where $\epsilon_m > 0$ are small enough so that $B(1/k_{m+1}, \epsilon_m)\subset 1/A_{m+1}$ and $1/B(1/k_{m+1}, \epsilon_m)\subset B(k_{m+1}, R/4)$ for all $m\in\mathbb{N}$.

Defining $f(z) = g(z) + 1/z$, we see immediately that $f$ is meromorphic with a single pole at the origin. From the last inequality, $f$ approximates the translation $z\mapsto z + 5$ in the disc $B_-$; by Rouché's theorem, it follows that $f$ has a zero in $B_-$, and thus $f$ is not in class $\mathcal{P}$. Furthermore, the first three inequalities imply that, for all $m\in\mathbb{N}$,
\begin{equation} \label{eq:inclusions}
 \begin{cases}
        f(A_m) \subset 1/A_{m+1},\\
        f^2(A_m) \subset A_{m+1},\\
        f(B_+)\subset B_+,\\
        f(B_m)\subset B_+
   \end{cases}
\end{equation}
By Montel's theorem, $B_+$ is contained in an invariant domain $V$ of $F(f)$, and each $B_m$ belongs to a pre-image of $V$. Likewise, each $A_m$ is contained in a Fatou component $U_{2m}$ with $U_{m+1}$ containing $f(A_m)\subset 1/A_{m+1}$, and upon proving that the $U_m$ are disjoint we can conclude that $U_m$ is in fact a ping-pong orbit of wandering domains. To be precise, we can only prove that all but finitely many of them are distinct, but that is enough for our conclusion to hold.

To this end, we claim that if $U_m = U_n$ for some $m\neq n$, then $U_m$ is multiply connected. Indeed, since $F$ and the inequalities (\ref{eq:DavidIneq}) are all symmetric with respect to $\mathbb{R}$, it follows from a symmetric version of Nersesyan's theorem -- see \cite[~Section 2]{Gau13} -- that $g$, and thus $f$, can be taken to be symmetric w.r.t. $\mathbb{R}$, and so any component of $F(f)$ is also symmetric with respect to the real line. Since $U_m$ must contain both $k_m$ and $k_n$ but avoid any $B_l$ between them, it follows that $U_m$ must be multiply connected. If this happens infinitely often, then $U_1\supset A_1$ is a Baker wandering domain by Lemma \ref{lem:RSBaker}, which means that $U_n$ escapes to infinity as $n\to +\infty$. However, we know that $U_n$ returns to the unit disc every second iterate, contradicting the fact that it is a Baker wandering domain. This concludes our construction.

In our construction, we took great care to guarantee that $f$ was not in class $\mathcal{P}$; with a few modifications, however, the same method can be used to obtain an example in this class.

First, one modifies the sets $A_m$ and $B_m$: instead of two families of circles, we require $A_m$ to be a family of symmetric (w.r.t. $\mathbb{R}$) annuli sectors instead. Notice that this means that $\log A_m$ is a rectangle (which can be chosen to be a square) centred at a point $a_m$. Next, we need $f$ to be of the form $f(z) = \exp\left(g(z)\right)/z$; in order to write the inequalities concerning $g$, we consider $\log f(z) = g(z) - \log z$. They take the form:
\[ \begin{cases}
    |g(z) - \log z - a_m| < K_1, & z\in A_m \\
    |g(z) - \log z - \log 2| < K_2, & z\in B_+\cup\left(\bigcup_{m\in\mathbb{N}} B_m\right) \\
    |g(z)| < K_3, & z\in \mathbb{D}
   \end{cases}, \]
and the positive constants $K_1$, $K_2$ and $K_3$ are chosen so that the inclusions (\ref{eq:inclusions}) hold as before. From here, the same arguments as above can be used to show that $f(z) = \exp\left(g(z)\right)/z$ has a ping-pong wandering domain.

\subsection{Functions in class $\mathcal{P}$} \label{ssec:P}
We now prove Theorem \ref{thm:pingpong} in the case when $f\in\mathcal{P}$. In this case, we appeal to a result of Martí-Pete \cite{Mar18} on the escaping sets of transcendental self-maps of $\Cx^*$, by which we mean a holomorphic map $f:\Cx^*\to\Cx^*$ such that both $0$ and $\infty$ are essential singularities. In order to state it, we shall need some definitions.

We call any sequence $(e_n)_{n\in\mathbb{N}}\in\{0, \infty\}^\mathbb{N}$ an essential itinerary. Given a holomorphic function $f:\Cx^*\to\Cx^*$ and any $r > 0$, its maximum and minimum moduli are denoted, respectively, by
\[ \text{$M(r, f) = \max_{|z| = r} |f(z)|$ and $m(r, f) = \min_{|z| = r} |f(z)|$}. \]
\begin{definition}
Let $f$ be a transcendental self-map of $\Cx^*$ and $e = (e_n)_{n\in\mathbb{N}}$ an essential itinerary. Let $R > 0$ and define $R_1 = R$ if $e_1 = \infty$ or $R_1 = 1/R$ if $e_1 = 0$. For $n > 1$, set
\[
    R_n = \begin{cases} M(R_{n-1}, f) & \text{if $e_n = \infty$} \\
                        m(R_{n-1}, f) & \text{if $e_n = 0$} \end{cases}, \]
and assume that $R$ was chosen large enough that the sequence $R_n$ accumulates at $\{0, \infty\}$. For $l\in\mathbb{Z}$, we define $A_e^{-l}(R, f)$ to be the set of $z\in\Cx$ such that
\[ \text{$|f^{n+l}(z)| \geq R_n$ if $e_n = \infty$ and $|f^{n+l}(z)| \leq R_n$ if $e_n = 0$} \]
for all $n\in\mathbb{N}$ satisfying $n + l\in\mathbb{N}$. Finally, the \textit{fast escaping set with respect to the essential itinerary} $e$ is
\[ A_e(f) = \bigcup_{l\in\mathbb{Z}} \bigcup_{k\in\mathbb{N}} A_{\sigma^k(e)}^{-l}(R, f), \]
where $\sigma$ denotes the Bernoulli shift map.
\end{definition}

As suggested by the notation, $A_e(f)$ is independent of our choice of $R$ as long as it satisfies certain conditions specified in \cite{Mar18}, and which shall not be our concern. We are now ready to state the needed result.

\begin{lemma} \label{lem:david}
Let $f$ be a transcendental self-map of $\Cx^*$. Then, for any essential itinerary $e\in\{0, \infty\}^\mathbb{N}$, $A_e(f)\cap J(f) \neq \emptyset$ and $J(f) = \partial A_e(F)$.
\end{lemma}

We are ready to prove Theorem \ref{thm:pingpong} for functions in class $\mathcal{P}$; we start with a function $f$ in this class, and assume (conjugating $f$ by a translation if necessary) that its only pole is at the origin. Then, $f^2$ is a transcendental self-map of $\Cx^*$, and we apply Lemma \ref{lem:david} with the essential itinerary $e = (e_n)_{n\in\mathbb{N}}$ chosen as $e_n = 0$ if $n$ is even and $e_n = \infty$ if $n$ is odd. Then, if $z_0\in A_e(f)$ and $l$ is as required in the definition of $A_e(f)$, it follows that there exists $z$ in the orbit of $z_0$ such that
\[ \text{$(f^2)^{2n + l}(z)\to 0$ and $(f^2)^{2n + 1 + l}(z)\to\infty$ as $n\to+\infty$.} \]
It is clear, then, that $z$ has a ping-pong $f$-orbit for the subsequences $m_k = 2(2k + l)$ and $n_k = 2(2k + 1 + l)$. The density of such points in $J(f)$ follows from the fact that $J(f^2) = \partial A_e(f^2)$, by Lemma \ref{lem:david}.

\subsection{Functions not in class $\mathcal{P}$} \label{ssec:notP}
Finally, we prove Theorem \ref{thm:pingpong} when $f\notin\mathcal{P}$. For this case, we shall rely on the following version of the Ahlfors Five Islands Theorem \cite[~Lemma 5]{Ber93}.
\begin{lemma} \label{lem:AFIT}
Let $f$ be a transcendental meromorphic function with finitely many poles, and $D_1$, $D_2$ and $D_3$ be three simply connected domains in $\Cx$ with disjoint closures. Then, there exists $j\in\{1, 2, 3\}$ and, for any $R > 0$, a simply connected domain $G\subset \{z\in\Cx : |z| > R\}$ such that $f$ maps $G$ conformally onto $D_j$.
\end{lemma}

We shall also need the following lemma -- see Rippon and Stallard \cite[~Lemma 1]{RS11} for this particular statement and proof.

\begin{lemma} \label{lem:RSslow}
Let $E_n$, $n\in\mathbb{N}$, be a sequence of compact subsets of $\Cx$ and $f:\Cx\to\Chat$ be a continuous function such that
\[ \text{$f(E_n)\supset E_{n+1}$ for all $n\in\mathbb{N}$.} \]
Then, there exists $z\in\Cx$ such that $f^n(z)\in E_n$ for all $n\in\mathbb{N}$. If, in addition, $f$ is meromorphic and $E_n\cap J(f)\neq \emptyset$ for $n\in\mathbb{N}$, then $z$ can be chosen to lie in $J(f)$.
\end{lemma}

We are ready to start the proof of Theorem \ref{thm:pingpong}; let $f$ be a transcendental meromorphic function with finitely many poles and not in class $\mathcal{P}$, and pick $p\in\Cx$ a pole of $f$ which is not an omitted value. Using Lemma \ref{lem:AFIT}, we shall construct an appropriate sequence of compact subsets $E_n$ and then apply Lemma \ref{lem:RSslow} to obtain a ping-pong orbit around $p$; we refer to Figure \ref{fig:pingpong} for a representation of the process.

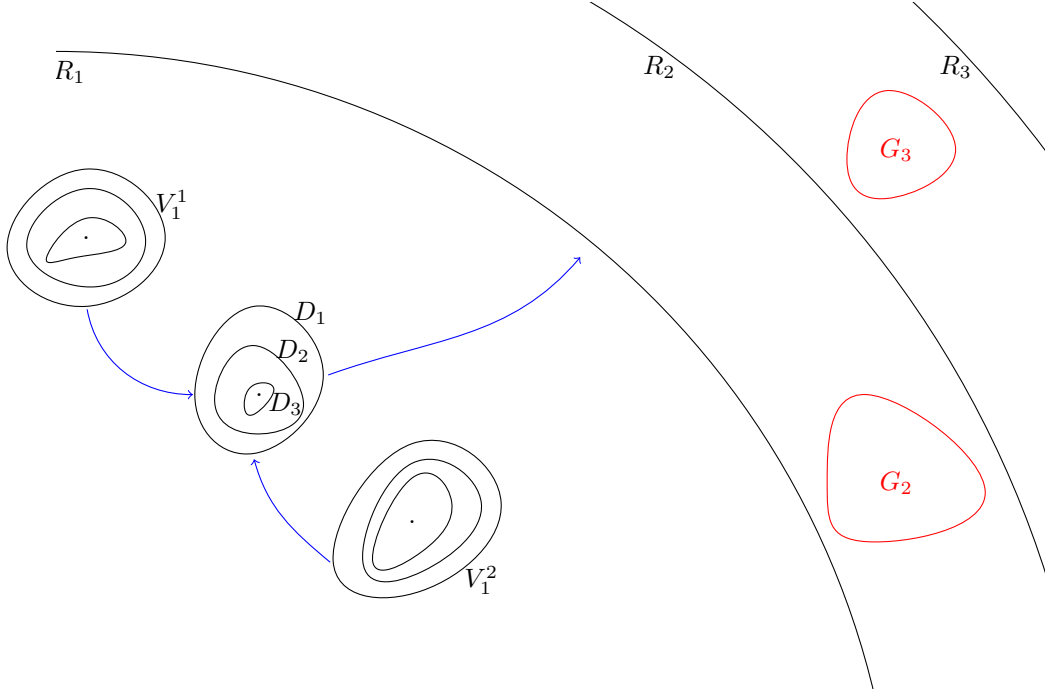
\begin{figure}[!h]
    \centering
    \begin{tikzpicture}[scale=1.3] 

        \begin{scope}
        \clip(0,-2) rectangle ++(10,7);

        \draw (0,-4) circle [radius=8.5];
        \node[anchor=north] at (0.13,4.5) {$R_1$};

        \draw (0,-4) circle [radius=10.5];
        \node[anchor=north] at (6.1,4.55) {$R_2$};

        \draw (0,-4) circle [radius=12.5];
        \node[anchor=north] at (9.1,4.55) {$R_3$};
        \end{scope}

        \draw plot[smooth cycle,tension=0.9] coordinates {(2,0.8) (2.2,1.05) (2,1.1) (1.9,0.9)};
        \node at (2.05,1) {$.$};
        \node[anchor=west] at (2.05,0.9) {$D_3$};

        \draw plot[smooth cycle,tension=0.9] coordinates {(2,0.6) (2.5, 0.9) (2, 1.5) (1.6,0.95)};
        \node at (2.39,1.45) {$D_2$};

        \draw plot[smooth cycle,tension=0.9] coordinates {(2,0.4) (2.7,1.2) (2,1.9) (1.4,1)};
        \node at (2.57,1.83) {$D_1$};


        \draw plot[smooth cycle,tension=0.9] coordinates {(3.7,-1) (4.5,-0.1) (3.6,0.5) (2.8,-0.7)};
        \node at (3.6,-0.3) {$.$};
        \node[anchor=north] at (4.3,-0.65) {$V_1^2$};

        \draw plot[smooth cycle,tension=0.9] coordinates {(3.71,-0.8) (4.3,-0.11) (3.6,0.3) (3.1,-0.7)};

        \draw plot[smooth cycle,tension=0.9] coordinates {(3.7,-0.6) (4,-0.1) (3.6,0.15) (3.2,-0.7)};


        \draw plot[smooth cycle,tension=0.9] coordinates {(0.3,1.9) (1.1,2.6) (0.3,3.3) (-0.5,2.5)};
        \node at (0.3,2.6) {$.$};
        \node[anchor=south] at (1.17,2.7) {$V_1^1$};

        \draw plot[smooth cycle,tension=0.9] coordinates {(0.4,2.1) (0.9,2.6) (0.3,3.1) (-0.3,2.5)};

        \draw plot[smooth cycle,tension=0.9] coordinates {(0.2,2.4) (0.7,2.55) (0.31,2.8) (-0.1,2.4)};


        \draw[->,blue] (2.75,1.2) to [out=20,in=230] (5.3,2.4);

        \draw[->,blue] (0.31,1.87) to [out=280,in=180] (1.38,1);

        \draw[->,blue] (2.77,-0.71) to [out=140,in=285] (2,0.34);


        \draw[red] plot[smooth cycle,tension=0.9] coordinates {(8.2,-0.5) (9.4,0) (8.2,1) (7.8,0.1)};
        \node[red] at (8.5,0.1) {$G_2$};

        \draw[red] plot[smooth cycle,tension=0.9] coordinates {(8.4,3) (9.1,3.5) (8.4,4.1) (8,3.4)};
        \node[red] at (8.5,3.5) {$G_3$};

        \end{tikzpicture}
    \caption{A schematic of the sets involved in the construction of a ping-pong orbit. In red, we highlight the domains $G_n$ obtained through Ahlfors's Five Islands Theorem. The blue arrows denote mapping under $f$.}
    \label{fig:pingpong}
\end{figure}

First, pick any positive, monotone sequence $R_n\to + \infty$; we say any sequence, but we reserve the right to ask that $R_1$ be large enough to satisfy certain constraints to be specified further ahead. Then, take a simply connected neighbourhood $D_1$ of $p$ such that $f:D_1\to\{z\in\Cx : |z| > R_1\}$ is a covering map, possibly branched at $p$ (in order to guarantee that $D_1$ can satisfy such conditions, we have already asked that $R_1$ be large). Since $p$ is not omitted, we can choose bounded, simply connected pre-images $V_1^1$ and $V_2^2$ of $D_1$ so that $f:V_1^i\to D_1$ are also possibly branched covering maps\footnote{For convenience, all covering maps henceforth are taken to be possibly branched.}. As preparation to use Lemma \ref{lem:AFIT}, we require that $R_1$ is large enough that $D_1$, $V_1^1$ and $V_1^2$ have disjoint closures.

Next, we pick $D_2\subset D_1$ so that $f:D_2\to\{z\in\Cx : |z| > R_2\}$ is a covering map, and pre-images $V_2^i\subset V_1^i$ of $D_2$, $i=1, 2$, following the same ``recipe'' as before. Now, we use Lemma \ref{lem:AFIT} to find a simply connected $G_2\subset \{z\in\Cx : |z| > R_1\}$ such that $f$ maps $G_2$ conformally onto one of $D_2$, $V_2^1$ or $V_2^2$.

After coming back to $D_2$, the next step is to choose $D_3\subset D_2$ so that $f:D_3\to\{z\in\Cx : |z| > R_3\}$ is a covering map, and pre-images $V_3^i\subset V_2^i$ as before. Again by Lemma \ref{lem:AFIT}, we find a domain $G_3\subset \{z\in\Cx : |z| > R_2\}$ such that $f$ maps $G_3$ conformally onto one of $D_3$, $V_3^1$ or $V_3^2$.

Proceeding inductively, we produce sequences of domains $D_n$, $G_n$ such that, for all $n$,
\begin{enumerate}[(i)]
    \item $D_n$ are progressively smaller neighbourhoods of $p$;
    \item $G_{n+1}\subset \{z\in\Cx : |z| > R_n\}$;
    \item $f(D_n)\supset G_{n+1}$;
    \item either $f$ maps $G_n$ conformally onto $D_n$, or $f$ maps $G_n$ conformally onto a covering space of $D_n$ (i.e., $V_n^i$ for some $i\in\{1, 2\}$).
\end{enumerate}

Therefore, we can construct a sequence of compact sets $E_n$ satisfying the hypotheses of Lemma \ref{lem:RSslow} by alternating between $\overline{D_n}$ and $\overline{G_n}$, going through $\overline{V_n^i}$ between $\overline{G_n}$ and $\overline{D_n}$ as necessary. By Lemma \ref{lem:RSslow}, there exists $z_0\in\Cx$ such that $f^n(z_0)\in E_n$ for all $n\in\mathbb{N}$, and so by our choice of $E_n$ the point $z_0$ has a ping-pong orbit. Furthermore, since $p\in J(f)$, each $D_n$ intersects the Julia set, and by the complete invariance of $J(f)$ this means that every $G_n$ and $V_n^i$ do too. Thus, by Lemma \ref{lem:RSslow}, $z_0$ can be chosen to be in the Julia set.

Finally, to prove that points with ping-pong orbits are dense in $J(f)$, notice that if $z$ is a point with a ping-pong orbit then any $f^n(z)$ also has a ping-pong orbit, and so does any point in $f^{-n}(z)$. Therefore, the set $BU_P(f)$ of points with a ping-pong orbit is completely $f$-invariant, and thus $\overline{BU_P(f)\cap J(f)}$ is a closed, completely invariant set with more than three points -- since it contains $z_0$ and its entire grand orbit. By the minimality of the Julia set (which follows from Montel's theorem; see \cite[~p. 67]{Bea91} for the case of rational functions), $J(f)\subset \overline{BU_P(f)\cap J(f)}$ so $BU_P(f)$ is dense in $J(f)$.

\end{document}